\documentclass[12pt]{amsart}
\newtheorem{theorem}{Theorem}[section]
\newtheorem{lemma}[theorem]{Lemma}
\newtheorem{corollary}[theorem]{Corollary}
\newtheorem{proposition}[theorem]{Proposition}
\newtheorem{problem}[theorem]{Problem}

\theoremstyle{definition}

\theoremstyle{remark}
\newtheorem{remark}[theorem]{Remark}
\newenvironment{theoremno}[2]{\begin{trivlist}
\item[\hskip \labelsep {\bfseries #1}\hskip \labelsep {\bfseries #2}]}{\end{trivlist}}

\numberwithin{equation}{section}

\newcommand{\sai} {\mbox{$\to \kern -0.50 em \to$}}
\newcommand{\nsai} {\mbox{$\not\to \kern -0.50 em \to$}}
\newcommand{\st} {\stackrel}
\newcommand{\hr} {\hookrightarrow}

\begin{document}

\title[Geometry of the Banach spaces $C(\beta {\mathbb N} \times   K,
X)$ \today] {Geometry of the Banach spaces  $C(\beta {\mathbb N} \times K, X)$ for   compact metric spaces $K$}

\author{Dale E. Alspach
}
\address{Department of Mathematics, Oklahoma State University, Stillwater OK 74078, USA}
\email{alspach@math.okstate.edu}

\author{El\'oi Medina Galego}
\address{Department of Mathematics, University of S\~ao Paulo, S\~ao Paulo,
Brazil 05508-090          }
\email{eloi@ime.usp.br}





\begin{abstract} { A classical result of Cembranos and Freniche states
that the $C(K, X)$ spaces  contains a complemented copy of $c_{0}$ whenever
$K$ is an infinite compact Hausdorff space and $X$ is an infinite dimensional
Banach space. This paper takes this result as a starting point and
begins a study of the conditions under which the spaces
$C(\alpha)$, $\alpha<\omega_1$ are quotients of or complemented in spaces $C(K,X)$.

In contrast to the $c_0$ result,
we prove that if $C(\beta \mathbb N \times [1,\omega], X)$ contains
a complemented copy of  $C(\omega^\omega)$ then $X$ contains a copy
of $c_{0}$. Moreover,  we show that  $C(\omega^\omega)$ is not even
a quotient of $C(\beta {\mathbb N} \times [1,\omega], l_{p})$, $1<p<
\infty$.

We then completely determine the separable $C(K)$ spaces which are
isomorphic to a complemented subspace or a  quotient of the $C(\beta
{\mathbb N} \times [1,\alpha], l_{p})$ spaces for countable
ordinals $\alpha$ and  $1 \leq p< \infty$. As a consequence, we
obtain the isomorphic classification of the $C(\beta {\mathbb N}
\times K, l_{p})$ spaces for infinite  compact metric spaces $K$ and  $1
\leq p  < \infty$. Indeed, we establish  the following more general
cancellation law. Suppose  that the Banach space  $X$ contains
no copy of $c_{0}$ and $K_{1}$ and $K_{2}$ are infinite compact
metric spaces,  then the following statements are equivalent:
\begin{enumerate} \item [(1)] $C(\beta \mathbb N \times K_{1}, X)$ is
isomorphic to   $C(\beta \mathbb N \times K_{2}, X)$ \end{enumerate}
\begin{enumerate} \item [(2)] $C(K_{1})$ is isomorphic to  $C(K_{2}).$
\end{enumerate} These results are applied to the isomorphic classification of
some spaces of compact operators.
} \end{abstract}

\maketitle




\par

\section{Introduction}
The isomorphic classification of the separable spaces of continuous functions
on a compact Hausdorff space was completed in 1966 when Milutin, \cite {M}, \cite{R}, showed that
there was a single isomorphism class for the continuous functions on
uncountable compact metric spaces. For general compact Hausdorff spaces some
work has been done in special cases, e.g., \cite{GulO} or \cite{K}, but, unlike the isometric case
which is completely determined by 
the Banach-Stone theorem,  extended in \cite{Am} and \cite {Cam}, the
isomorphic classification seems hopeless.

In this paper we consider a special class of compact Hausdorff spaces but allow the range space to be a
Banach space instead of $\mathbb R.$ Thus we study the spaces $C(K,X)$ of continuous functions from
$K$ into $X$
where $X$ is a Banach space,
$K$ is a compact Hausdorff space and the norm of an element $f$ is
$\|f\|=\sup_{k\in K} \|f(k)\|_X.$ Usually $K$ will be a compact metric space
or a product of a compact metric space and the Stone-Cech compactification of
the natural numbers, $\beta \mathbb N.$ Our interest in $\beta \mathbb N$
stems from application of some of the results to the following question,
\cite[Problem 4.2.2]{G2}. From now on  ${\mathcal
K}(X,Y)$ denotes the space of compact operators from $X$ to  another Banach
space $Y$ and $[1,\alpha]$ is the compact Hausdorff space of ordinals between
$1$ and $\alpha$ in the order topology.

\begin{problem}\label{pp} Classify, up to an isomorphism, the spaces of compact operators ${\mathcal
K}(l_{1}, C([1, \alpha], l_{p}))$, where $\alpha \geq \omega$ and $1 \leq p< \infty$.
\end{problem}

This problem covers some cases remaining from
the development in \cite{G2}, \cite{G3},
\cite{G4}, \cite{G5} and \cite{S2}, of the isomorphic classification of
some spaces of compact operators. 
We give the solution to above problem
in the case where  $\alpha$ is countable. The connection to the spaces
$C(K,X)$ comes through the injective tensor product. Notice that
that since $l_{1}$ has the approximation property, by \cite[Proposition 5.3]{DF}
we know that for every ordinal $\alpha$ and  $1\leq p < \infty$,
$${\mathcal K}(l_{1}, C([1, \alpha], l_{p})) \sim C(\beta \mathbb N \times
[1,\alpha], l_{p}).$$

The notation for the spaces is a bit cumbersome so we will shorten some expressions. When
the context clearly requires a compact Hausdorff space we will write $\alpha$
rather than $[1,\alpha]$. In particular, $C(\alpha, X)= C([1,\alpha],X).$ If $X=
\mathbb R$, we will write $C(K)$ rather than $C(K,\mathbb R).$ 
We will also adopt some standard notational conventions from Banach space theory. We write $X \sim Y$
when the Banach spaces $X$ and $Y$ are isomorphic, $Y \hr X$ when $X$
contains a copy of $Y$, that is, a subspace  isomorphic to  $Y$,
$Y \st{c}{\hr} X$ if $X$ contains a complemented copy of $Y$ and
$X \sai Y$ when  $Y$ is a quotient of  $X$.  For other notation and
terminology we refer the reader to  \cite{JL} and  \cite{LT}.

In $C(K,X)$ an obvious part of the difficulty with the isomorphic classification is that structures in $X$
can be used to find an alternate compact Hausdorff space $K_1$ so that $C(K_1,X)$ is isomorphic to
$C(K,X)$. A second difficulty is that structures may arise that are present in neither $C(K)$ nor $X$.
Consider the following result,  which was obtained
independently by Cembranos \cite[Main Theorem]{C} and Freniche
\cite[Corollary 2.5]{F}.  

\begin{theorem}\label{CF} Let $K$ be an
infinite compact Hausdorff space and  $X$  an infinite dimensional Banach
space. Then $$c_{0} \st{c}{\hr} C(K, X).$$ \end{theorem} 
Consequently, both $C(\beta \mathbb N, C(\beta \mathbb N))$ and $C(\beta \mathbb N,\ell_2)$ contain a complemented
subspace isomorphic to $c_0$ despite the fact that neither $C(\beta \mathbb
N)$ nor $\ell_2$ have complemented copies of $c_0$.

It is natural
to ask if there are other $C(K)$ spaces for which the analogous result holds. The next more complicated $C(K)$
space is $C(\omega^\omega)$. Our first
result gives a negative answer in this case. We prove  that
even when $C(K)$ contains 
a complemented copy of $c_{0}$ and $X$ is an infinite dimensional
Banach  space,  $C(K, X)$ may not contain a complemented copy of
$C(\omega^\omega)$. Indeed, it is easy to see that
$C(\beta {\mathbb N} \times
\omega)$ contains a complemented copy of $c_{0}$. However,  in
Section \ref{compCK}, we  prove the following.  

\begin{theoremno}{Theorem}{\ref{com}.} Let $X$
be a Banach space. Then $$C(\omega^\omega)  \st{c}{\hr} C(\beta \mathbb
N \times \omega , X) \Longrightarrow c_{0} \hr X.$$ \end{theoremno}

We then extend Theorem \ref{com} to larger ordinals by using that result
and the structure of the ordinals.

\begin{theoremno}{Theorem}{\ref{excom}.} Let $X$ be a Banach space  containing no copy of $c_{0}$, $K$ an infinite compact metric space and $0 \leq \alpha< \omega_{1}$. Then
$$C(K)   \st{c}{\hr}     C(\beta {\mathbb N}  \times 
\omega^{\omega^\alpha}, X)    \Longleftrightarrow  C(K) \sim
C(\omega^{\omega^\xi}) \ \hbox{for some} \ \ 0 \leq \xi \leq \alpha.$$
\end{theoremno} 

In Section \ref{quotients} we turn our attention to  spaces of the form $C(\beta
{\mathbb N} \times \alpha,X)$ where $X$ satisfies some geometrical
properties, $(\dagger)$ and $(\ddagger)$,
that are modeled on simple properties of $\ell_p$, $1\le p <\infty.$ In
particular we get that $C(\omega^\omega)$ is not a quotient of $C(\beta
\mathbb N \times  \omega, l_{p})$,  $1<p< \infty$.

\begin{theoremno}{Theorem}{\ref{omomX}.} Suppose that $X$ is a Banach space
satisfying the daggers. Then $C(\omega^\omega)$
is not a quotient of  $C(\omega \times \beta {\mathbb N}, X).$
\end{theoremno}

Similar to the way we obtained Theorem \ref{excom} from Theorem \ref{com}
in Section  \ref{quotients}, we also extend
Theorem \ref{omomX} by proving.

\begin{theoremno}{Theorem}{\ref{exquo}.} Let $K$ be an infinite compact
metric space and $0 \leq \alpha < \omega_{1}$. Then $$C(\beta {\mathbb N}
\times \omega^{\omega^\alpha}, l_{p}) \sai  C(K)  \Longleftrightarrow
C(K) \sim C(\omega^{\omega^\xi}) \ \hbox{for some} \ \ 0 \leq \xi \leq
\alpha.$$ \end{theoremno}

The next section concerns the isomorphic classification of the  spaces $C(\beta {\mathbb N} \times \alpha, l_{p})$.  Our results   also provide us with immediate information
about the  isomorphic classifications of  a wider class of Banach
spaces. Namely,  the $C(\beta  {\mathbb N} \times K, X)$ spaces, where
$X$ contains no copy of $c_{0}$ and $K$ is a metrizable compact space,
that is, $C(K)$ is a separable space. Indeed, in Section \ref{compCK}
we prove the following cancellation law which is the
main application of the results of the paper. 
The case $X=l_{p}$, $1 \leq p< \infty$, gives the solution to Problem \ref{pp}.
\begin{theoremno}{Theorem}{\ref{main}.} Let $X$ be a Banach space
containing no copy of $c_{0}$. Then for any infinite compact
metric spaces $K_{1}$ and $K_{2}$ we have
$$C(\beta {\mathbb N} \times K_{1}, X) \sim C(\beta {\mathbb N} \times K_{2}, X) \Longleftrightarrow C(K_{1}) \sim C(K_{2}).$$
\end{theoremno}

Moreover, in Section \ref{Clpom}, we accomplish the isomorphic classification of the spaces  $C(\beta  {\mathbb N} \times K, l_{p})$ by considering also the case where $K$ is finite.  In order to do this, we first prove  a general result about the spaces of
compact operators $\mathcal K(\ell_p(X),\ell_q(Y))$ (Theorem \ref{isok}).   From that we deduce the following.
\begin{theoremno}{Theorem}{\ref{iso}.} Let $1\leq p< \infty$. Then
$$C(\beta \mathbb N \times \omega, l_{p}) \sim C(\beta {\mathbb N}, l_{p}).$$
\end{theoremno}

Finally, in Section \ref{prob}, we pose some elementary questions which this work raises.

\section{Preliminaries}\label{Sec 2}

In this section we recall some results that we will use in the sequel. 

In 1920 Mazurkiewicz and
Sierpi\'nski  showed that if $K$ is an countable compact metric space then it is
homeomorphic to an interval of ordinals $[1, \alpha]$ with $\omega \leq
\alpha< \omega_{1}$ \cite{MS}. This was used in 
the isomorphic classification of the $C(\alpha)$ spaces,
$\omega  \leq \alpha< \omega_{1}$,  obtained in 1960
by Bessaga and Pe\l czy\'nski. They  showed that  if   $\omega \leq \alpha \leq \beta< \omega_{1}$ then   $C(\alpha)$ is isomorphic to $C(
\beta)$ if and only if
$\beta < \alpha^\omega$, see  \cite{BP2} and  \cite{R}. In particular this means
that the spaces $C(\omega^{\omega^\gamma})$, for $0 \le \gamma <\omega_1$, are a
complete set of representatives of the isomorphism classes of $C(K)$ where $K$ is
a countably infinite, compact metric space. 

Bessaga and Pe\l czy\'nski
actually prove some things for the case of $C(K,X),$ where $X$ is a
Banach space.

\begin{proposition} \label{BPiso}
Suppose $X$ is a Banach space and $\alpha$ is an infinite ordinal. Then  $C({\alpha},X)$ is isomorphic to 
\begin{enumerate}
\item{}  $C_0({\alpha},X)=\{f \in
C({\alpha},X): f(\alpha)=0 \}$

\leftline{and to}

\item{}  $C({\omega^{\omega^\beta}},X)$
whenever $\omega^{\omega^\beta} \le \alpha <\omega^{\omega^{\beta+1}}.$
\end{enumerate}
\end{proposition}

Thus for some Banach spaces $X$ there may be fewer isomorphism classes for the
spaces $C(K,X)$ with $K$ countable, compact metric, than for
the case $X=\mathbb R.$ That is what happens for Banach spaces which are
isomorphic to their squares or to the $c_0$-sum of infinitely many copies
of the space. Indeed, for $\ell_p$, $1\le p <\infty$, the finite ordinals all yield
the same space; for $c_0$ all of the spaces $C(\alpha,c_0)$, $\alpha
<\omega^\omega$, are isomorphic.

\begin{remark}\label{pro} The order structure on the spaces of ordinals make it easy to find
contractively complemented
subspaces of $C(\omega^\alpha)$ isometric to $C(\omega^\beta)$ for
$\beta<\alpha$. Indeed if $A$ is a closed subset of $[1,\omega^\alpha]$
and $A^{(1)}$ is the set of non-isolated points of $A$,
we can define a subspace $Y$
of $C(\omega^\alpha)$ isometric to $C(A)$ by
\begin{multline*}
Y=\{f \in C(\omega^\alpha) : f(\gamma)=f(\xi), \  \forall \gamma \text{
such that }\\ \sup \{\rho<\xi:\rho \in A\}<\gamma<\xi \text{ and } \xi \in
A \setminus A^{(1)}\}.\end{multline*}
We can define a projection onto $Y$ by restricting to $A$ and then
extending by the formula in the definition of $Y$, i.e.,
\begin{multline*}
Lg(\gamma)=g(\xi)\text{ for all
}\gamma \text{
such that } \\ \sup \{\rho<\xi:\rho \in A\}<\gamma<\xi \text{ and } \xi \in
A \setminus A^{(1)}.\end{multline*}

For $\gamma>\sup A$ let $Lg(\gamma)=0.$
\end{remark}

The spaces $c_0$ and $C(\beta \mathbb N)$ play a prominent role in this
paper so we now recall some important properties of these spaces. Bessaga and Pe\l czy\'nski made a study of $c_0$ in \cite{BP1} and introduced the
notion of a weakly unconditionally Cauchy sequence (w.u.c.). A sequence
$(x_n)$ in  a Banach space $X$
is said to be a w.u.c. if and only if for every
$x^* \in X^*,$ $\sum_n |x^*(x_n)| < \infty.$ A sequence equivalent to
standard basis of $c_0$ is clearly a w.u.c. We will use the following
result from their paper.

\begin{proposition}
Suppose that $X$ is a Banach space which has no subspace isomorphic to
$c_0$ then every w.u.c. in $X$ is unconditionally converging. Consequently,
if $(x_n)$ is a w.u.c. in $X$, $\lim_n \|x_n\| = 0.$
\end{proposition}

$C(\beta
\mathbb N)$ is isometric to $\ell_\infty= \ell_\infty(\mathbb N)$, the
space of bounded sequences with the supremum norm. For any non-empty index
set $\Gamma$, $\ell_\infty(\Gamma)$ is injective, i.e., it is complemented
in any space which contains it. $c_0$ is separably injective, i.e., it is
complemented in any separable Banach space that contains it. $c_0$ is not
complemented in $\ell_\infty$ and in fact the only infinite dimensional
complemented subspaces  of $c_0$ or $\ell_\infty$ are isomorphs of the
whole space, \cite[pages 54 and 57]{LT}. $\ell_\infty$ is an example of a Grothendieck space,
i.e., any
weak${}^*$ convergent sequence in the dual is actually weakly convergent \cite[page 179]{DU}.
Actually this is the essential property of $C(\beta \mathbb N)$ that we use.
Except in one or two cases, e.g., Theorem \ref{BNextCF},
the results could be rewritten with $C(K)$ which
is Grothendieck space in place of $C(\beta \mathbb N)$.

While $C(\beta \mathbb N)$ and $\ell_\infty$ are isometric,  for many
infinite dimensional Banach space $X$,
$C(\beta \mathbb N,X)$ is not isomorphic to $\ell_\infty(X)=\{(x_n): x_n
\in X \text{ for all }n, \|(x_n)\|=\sup_n\|x_n\|_X<\infty\}.$  This is the
case if $X$ does not contain a complemented subspace isomorphic to $c_0$
since, by \cite{LR},  $\ell_\infty(X)$ only contains a complemented
subspace isomorphic to $c_0$ if $X$ does.

We identify $C(\beta\mathbb N, X)^*$ with the space
of $X^*$-valued regular Borel measures on $C(\beta\mathbb N)$ with
$\|\mu\|=\sup \sum_n \|\mu(A_n)\|_{X^*}$ where the supremum is over all
partitions of $\beta\mathbb N$ into disjoint clopen sets $\{A_n\}$,
\cite[page 182]{DU}.
Moreover if $(\mu_n)$
is a weak${}^*$ converging sequence of measures with limit $\mu$, then
for any clopen set $A$, $(\mu_n(A))$ converges weak${}^*$ in $X^*$
to $\mu(A).$

It will be convenient at times to shift point of view as to the underlying
compact Hausdorff space and the range space. Thus we will use the
fact that $C(K_1 \times K_2,X)$ is isomorphic to $C(K_1,C(K_2,X))$  and to
$C(K_2,C(K_1,X))$ where $K_1$
and $K_2$ are compact Hausdorff spaces. Also because $C(K,X)$ is isometric to the injective tensor product $C(K)\stackrel{\scriptscriptstyle
\vee}{\otimes}
X$, we may replace $K$ by $K_1$ if $C(K)$ is isomorphic to $C(K_1)$.

Let $\max_\sigma \{\beta_1,\beta_2\}$ be largest ordinal
$\beta=\gamma_1+\alpha_1+\gamma_2+\alpha_2+\dots
\gamma_k+\alpha_k$ obtained by writing $\beta_1=\gamma_1+\gamma_2 +
\dots +\gamma_k$ and
$\beta_2=\alpha_1+\alpha_2 + \dots +\alpha_k$, where $\gamma_j \ge 0$ and
$\alpha_j \ge 0$ for all $j$.
This can also be obtained by writing the ordinals $\beta_1$ and $\beta_2$
in terms of prime
components and arranging the terms of the sum in decreasing order. 

The topological results in \cite{MS} are based on the notion of derived
set. Recall that $K^{(0)}=K.$ For any ordinal $\alpha$,
$K^{(\alpha+1)}$ is the set of non-isolated
points in $K^{(\alpha)},$ and for a limit ordinal $\beta$, $K^{(\beta)}=
\cap_{\alpha<\beta} K^{(\alpha)}.$ We will only use this with countable
compact spaces and will refer to the smallest ordinal
$\alpha$ such that $K^{(\alpha)} \ne \emptyset$ and $K^{(\alpha+1)} =
\emptyset$ as the derived order of $K$.

\begin{lemma}\label{ordprod} Let $K_1$ and $K_2$
be countable compact metric spaces and
$\beta_1,\beta_2$ be countable ordinals such that $K_1^{(\beta_1)}$ and
$K_2^{(\beta_2)}$ are finite non-empty sets. Then $(K_1 \times
K_2)^{(\max_\sigma \{\beta_1,\beta_2\})}$ is a finite non-empty
set.
\end{lemma}

\begin{proof}[Sketch of Proof] First we can assume that $K_1^{(\beta_1)}$ and
$K_2^{(\beta_2)}$ are singletons, $k_1$ and $k_2$, respectively.
The proof is an induction on $\beta_2$ and for each $\beta_2$ on $\beta_1$,
$0\le \beta_1\le
\beta_2.$ The result is clear for $\beta_2 =0, 1$ and $\beta_1=0,1$ and for
all $\beta_2<\omega_1$ and $\beta_1=0.$ Assume $1<\beta_2$, $0<\beta_1\le
\beta_2$ and that the result holds for $K_2$ of order $\beta<\beta_2$ and
$K_1$ of order $\gamma\le \beta$
and for $\beta=\beta_2$ and $\gamma<\beta_1.$

To see that $$(K_1 \times
K_2)^{(\max_\sigma \{\beta_1,\beta_2\})}=\{(k_1,k_2)\},$$ notice that by the
induction assumption, if
$\alpha_1 <\beta_1$, $\alpha_2 <\beta_2$, $m_1 \in K_1^{(\alpha_1)}\setminus
K_1^{(\alpha_1+1)}$ and
$m_2 \in K_2^{(\alpha_2)}\setminus K_2^{(\alpha_2+1)}$, then
$$(m_1,m_2) \in (C_1 \times C_2)^{(\max_\sigma
\{\alpha_1,\alpha_2\})},$$ where $C_1$ and $C_2$, are apropriately chosen
clopen subsets of
$K_1$ and $K_2$,
respectively. For $i=1,2$ write $K_i\setminus \{k_i\}$,
as a disjoint union of clopen sets $C_{1,n}$,
derived order $\beta_i-1$ for all $n$ or derived order
$\beta_{i,n}$ and $(\beta_{i,n})$
increasing to $\beta_i$. Each set $C_{1,n} \times [1,\beta_2]$ and
$[1,\beta_1]\times C_{2,n}$ satisfies the induction hypothesis (possibly
symmetrized). Notice that $$\cup_n C_{1,n} \times [1,\beta_2] \cup
[1,\beta_1]\times C_{2,n}= K_1 \times K_2 \setminus \{(k_1,k_2)\}.$$
There are four cases to check. It is easy to see that
$$\max_\sigma\{\beta_1,\beta_2-1\}+1=\max_\sigma\{\beta_1,\beta_2\},$$ in the
first case, and
$\max_\sigma\{\beta_1,\beta_{2,n}\}$ increases to
$\max_\sigma\{\beta_1,\beta_2\}$ in the second case. The other two case are
similar.
\end{proof}

The  lemma shows that we do not really gain anything from simple
manipulations of compact metric spaces. Indeed, for countable ordinals
$\alpha$ and $\gamma$, we have that $$C(\beta \mathbb N \times
\omega^{\omega^\alpha}, C(\omega^{\omega^\gamma},X)) \sim 
C(\beta \mathbb N \times
\omega^{\omega^\alpha}\times \omega^{\omega^\gamma},X).$$ However
$$\omega^{\omega^{\max \{\alpha,\gamma\}}}\le \omega^{\max_\sigma \{\omega^\alpha,\omega^\gamma\}} \le
\omega^{\omega^{\max \{\alpha,\gamma\}}2}$$ and thus
$$C(\beta \mathbb N \times
\omega^{\omega^\alpha}\times \omega^{\omega^\gamma},X) \sim C(\beta \mathbb N \times
\omega^{\omega^{\max \{\alpha,\gamma\}}},X).$$

In a series of papers from the 1970's the first author developed some tools
for working with subspaces of $C(K)$ spaces isomorphic to $C(\alpha)$.
Some of the proofs in this paper are motivated in part by that work and
versions of some of the technical tools will be needed here. The first
is similar to \cite[Lemma 2.5]{A}.

\begin{lemma}\label{compress} Given a positive integer $k$ and $\epsilon>0$
there is a positive integer $n$
such that if $(x^*_\alpha)_{\alpha \le \omega^n}$ is a sequence in the
unit ball of the dual of a Banach space $X$ such that the function
$\alpha \rightarrow
x^*_\alpha$ is a order to weak${}^*$ continuous map, then there is a closed
subset
$B$ of $[1,\omega^n]$, order isomorphic and homeomorphic via the order
isomorphism to $[1,\omega^k]$ such that
for all $\beta, \beta' \in B,$ $|\|x^*_\beta\|-\|x^*_{\beta'}\||<\epsilon.$
\end{lemma}

The next lemma is a vector valued version of a typical construction of a
a sequence of disjointly supported functions normed by a sequence of
measures.

\begin{lemma} \label{disjoint_open} Suppose that $X$ is a Banach space, $C,D$ are positive
constants, $K$ is a compact
Hausdorff space, $(\mu_n)$ is a sequence of elements of $C(K,X)^*$
represented as $X^*$ valued measures on $K$, with $\|\mu_n\|\le C$ for all
$n$, and $(g_n)$ is a sequence of norm at most $D$
elements of $C(K,X)$ such that $$\int g_n \, d\mu_n \ge 1,$$ for all $n$ and
$\|g_n(t)\|\rightarrow 0$ as $n \rightarrow \infty$ for all $t \in K.$
Then for any $\epsilon >0$
there are an infinite subset $M$ of $\mathbb N$ and open subsets
$(G_m)_{m\in M}$
of $K$ such that  $\overline{G_m} \cap \overline{G_j} =\emptyset$ if $m \ne
j,$ and $$\int g_m 1_{G_m} \, d\mu_m > 1-\epsilon,$$ for all $m \in M.$
\end{lemma}

\begin{proof}
Let $\epsilon>0,$ and $\epsilon_k=\epsilon/2^{k+2},$ for $k \in \mathbb N.$
Choose $\rho>0$ such that $\rho <\epsilon/(4C).$ Then
$$\int g_n 1_{\{t:\|g_n(t)\|\ge \rho\}} \, d\mu_n >
1-\frac{\epsilon}{4}, \ \forall n, \ 1 \leq n< \omega.$$
Because $\|g_n(t)\|$ converges to $0$, $\int \|g_n(t)\| \,
d|\mu_i|(t) $ converges to $0$ for each $i$.

Thus for $i = 1,2, \dots, j_1$,
$j_1>C/\epsilon_1$,
there are infinite subsets $\mathbb N \supset N_1 \supset \dots \supset
N_{j_1}$, $n_1=1$, and $n_{i+1} =\min  N_i$
such that
$$\sum_{n \in N_i} |\mu_{n_i}|(\{t:\|g_n(t)\|>\rho/2\}) \le
\sum_{n \in N_i}(2/\rho) \int \|g_n(t)\| \,
d|\mu_{n_i}|(t) < \epsilon/(4C).$$
For $i=1,2,\dots, j_1,$ let
$$A_i= \{t:\|g_{n_i}(t)\|\ge\rho\} \setminus \cup_{n \in N_i}
\{t:\|g_n(t)\|>\rho\}.$$
Find disjoint open sets $H_1, H_2, \dots, H_{j_1}$ such that
$$\{t:\|g_i(t)\|> \rho/2\}\supset H_i\supset
A_i,$$ for each $i$. Because $\|\mu_m\|\le C$ for all $m$
and the sets $(H_i)$ are
disjoint, for some infinite subset $M_1$ of $N_{j_1}$ and some
$i_1$, $1\le i_1 \le j_1$, $$|\mu_m|(H_{i_1})<\epsilon/8=\epsilon_1,$$ for all $m \in M_1$.
Let $m_1=n_{i_1}$ and $G_{m_1}$ be an open set containing $A_{i_1}$ such that
$\overline{G_{m_1}} \subset H_{i_1}.$ $m_1$ is the first element of $M$ and
$\mu_{m_1}$ and $G_{m_1}$ are the corresponding
measure and open set pair.

Let $K_1=K\setminus H_{i_1}.$
Now notice that if we consider $(\mu_m|_{K_1})_{m \in M_1}$ and
$(g_m|_{K_1})_{m\in M_1}$  we have the original situation  with $1$ replaced
by $1-\epsilon/8$ as the lower bound on $$\int_{K_1} g_m \, d\mu_m.$$
Thus repeating the argument above with $\epsilon_2$
but choosing open sets as open subsets
of $K\setminus \overline{G_{m_1}}$ rather than $K_1$
(and hence open in $K$), we get
$\mu_{m_2}$ and $G_{m_2}$ with $G_{m_1} \cap G_{m_2} = \emptyset.$

Continuing in this way we can construct the required indices and open sets.

\end{proof}

\section{Complemented separable $C(K)$ subspaces of $C(\beta \mathbb N \times  \alpha, X)$}
\label{compCK}

It is clear that for any Banach space $X$,  $C(\beta \mathbb N \times \omega, X)$  contains a complemented copy of
$c_{0}$. This section is devoted to proving that $c_{0}$ is, up
to an isomorphism, the only separable $C(K)$ space which is complemented
in $C(\beta \mathbb N \times \omega, X)$  whenever $X$ contains no
copy of $c_{0}$.  This is a direct consequence of Theorem \ref{com}.

The next lemma is a
technical analog of a result of Bessaga and Pe{\l}czy\'nski,
see \cite[Theorem 4]{BP1}.
\begin{lemma}\label{linfty} Suppose that $X$ and $Y$ are Banach spaces and that
$T$ is an operator from $C(\beta \mathbb N,X)$ into $Y$. If there exist an
element $f$ of $C(\beta \mathbb N,X)$, $\delta>0$, a
sequence of disjoint non-empty
clopen subsets $(G_n)$ of $\beta \mathbb N$ and  a
sequence of $X^*$-valued measures $(\mu_n)$ contained in $T^{*}(B_{Y^*})$
such that
$$|\int f 1_{G_n} \, d\mu_n| >\delta,$$
for all $n$, then there is a
subspace $Z$ of $C(\beta \mathbb N,X)$ such that $Z$ is isomorphic to
$\ell_\infty$ and $T|_{Z}$ is an isomorphism into $Y$.
\end{lemma}
\begin{proof} Let $|\mu_n|$ denote the real-valued total variation measure
induced by
$\mu_n$. Observe that the sequence of pairs $(|\mu_n|,G_n)$ satisfy the
hypotheses of Rosenthal's disjointness \cite[Lemma 1, page 18]{R}.
Therefore there exists a subsequence $(|\mu_n|,G_n)_{n\in M}$ such that
for all $n \in M$, $$\sum_{j\in M, j\ne n}|\mu_n|(G_j) <\delta/(8\|f\|).$$

We also need to have $$|\mu_n|(\overline{\cup_{j\in M} G_j}
\setminus \cup_{j\in M} G_j)<\delta/(8\|f\|),$$
for all $n \in M$.
If for some $n\in M,$
$$|\mu_n|(\overline{\cup_{j\in M} G_j}
\setminus \cup_{j\in M} G_j) \ge \delta/(8\|f\|),$$
we can argue as follows.
Partition $M$ into an infinite number of infinite sets $(M_k)$. If for
some $k$, for all $n\in M_k$, $$|\mu_n|(\overline{\cup_{j\in M_k} G_j}
\setminus \cup_{j\in M_k} G_j)<\delta/(8\|f\|),$$ we can continue with $M_k$ in place
of $M$. If not, for each $k$ choose $n_k\in M_k$ such that
$$|\mu_{n_k}|(\overline{\cup_{j\in M_k} G_j}
\setminus \cup_{j\in M_k} G_j)\ge \delta/(8\|f\|).$$ Let $M^1=\{n_k:k \in \mathbb N\}.$
Observe that
for all $k$, $$(\overline{\cup_{j\in M_k} G_j}
\setminus \cup_{j\in M_k} G_j) \cap \overline{\cup_{j\in M^1} G_j}
=\emptyset.$$
Now if for all $n \in M^1$, $$|\mu_n|(\overline{\cup_{j\in M^1}
G_j}
\setminus \cup_{j\in M^1} G_j)<\delta/(8\|f\|),$$
we can use $M^1$ in place
of $M$. If not, notice that for all $n \in M^1,$ $$\|\mu_n
|_{\overline{\cup_{j\in M^1} G_j}}\|\le\|\mu_n\|-\delta/(8\|f\|).$$
We can split $M^1$ into infinitely many infinite sets and repeat the
previous argument. Each time this process reduces the norm of the
part of $\mu_n$ under consideration by $\delta/(8\|f\|)$.
Thus in at most
$\|f\|\ \|T\|8/\delta$ repetitions of the argument we will find the required infinite set $M$ such that for all $n \in M$, $$|\mu_n|(\overline{\cup_{j\in
M} G_j}
\setminus \cup_{j\in M} G_j)< \delta/8,$$
and $$\sum_{j\in M, j\ne
n}|\mu_n|(G_j) <\delta/(8\|f\|).$$

Let $Z$ be given by
$$\{g \in C(\beta \mathbb N, X):
g(t)=0 \ \ \forall  t\notin \overline{\cup_{n\in M} G_n}, g 1_{G_n} =c_n f 1_{G_n},
c_n \in \mathbb R
\ \ \forall  n \in M\}.$$
Because the range of $f$ is compact, for any
bounded sequence of real numbers $(c_n)_{n \in M}$, the function $h$ defined on
$\mathbb N$ by $$h(k)=\begin{cases} 0 &\text{ if }k\notin \cup_{m\in M} G_m,\\
c_nf(k) &\text{ if }k \in G_n\text{ and }n \in M. \end{cases}$$
is in $\ell_\infty(X)$ with relatively compact range
and hence extends continuously to
some function $H$ on $\beta \mathbb N$
with values in the symmetric radial hull of $\|(c_n)\|_\infty$ times the range of $f$.
Moreover because $\mathbb N$ is dense, the extension is unique and
must agree with $c_n f
1_{G_n}$ on $G_n$ for all $n \in M,$ and be $0$ on the closure of
$$\{k: k\in
\mathbb N, k
\notin \cup_{n\in M} G_n\}.$$
Therefore $Z$ is isomorphic to
$\ell_\infty,$ and for $(c_n)$ and $H$ as above,
$$(\delta/\|T\|)\|(c_n)_{n\in M}\|_\infty\le \inf_{n\in M} \|f 1_{G_n}\|\|(c_n)_{n\in
M}\|_\infty\leq \|H\|
\le \|f\|\|(c_n)_{n\in M}\|_\infty.$$

Continuing with the same notation,  we can get a lower bound on $\|T H\|$ as follows.
Observe that for each $n \in M$,
$$|\int H \, d\mu_n|$$
is greater than or equal to
$$|\int H 1_{G_n} \,
d\mu_n|- \sum_{j\in M, j\ne n}
|\mu_n|(G_j)\|f\||c_j|-|\mu_n|(\overline{\cup_{j\in M} G_j}\setminus
\cup_{j\in M} G_j) \|f\|\|(c_j)\|_\infty,$$
which in turn is greater than or equal to
$$|c_n| \ |\int f
1_{G_n} \, d\mu_n|- \|(c_j)_{j\in M}\|_\infty \delta/4 \ge \delta(|c_n|-\|(c_j)_{j\in
M}\|_\infty/4).$$
Taking the supremum over $n$ and noting that $\mu_n \in
T(B_{Y^*})$ completes the proof.
\end{proof}
\begin{theorem}\label{com} Let $X$
be a Banach space. Then $$C(\omega^\omega)  \st{c}{\hr} C(\beta \mathbb
N \times \omega , X) \Longrightarrow c_{0} \hr X.$$ \end{theorem}

\begin{proof} Assume that $X$ does not contain a subspace isomorphic to
$c_0$. We will show that the existence of the complemented subspace
isomorphic to $C(\omega^\omega)$
produces the
situation in the hypothesis of the previous
lemma. First we will reduce to a simplified situation. By Proposition \ref{BPiso}.1, $C(\beta \mathbb N \times 
\omega, X)$ is isomorphic to $C_0(\omega \times \beta \mathbb N, X), $ i.e.,
the $c_0$-sum of $C(\beta \mathbb N, X)$.
Assume  now that $T$ is a projection from $C_0(\omega \times \beta \mathbb N,
X)
$ onto a subspace $Y$ isomorphic to
$C(\omega^\omega)$.  Let $S:Y \rightarrow C(\omega^\omega)$ be the
isomorphism and suppose, without loss of generality, that $\|S\|\le 1.$
Then  $T^*S^*$ is an isomorphism with lower
bound some $\epsilon>0$, i.e., $$\|T^*S^*z\| \ge \epsilon \|z\|,$$
for all $z \in
C(\omega^\omega)^*.$ Choose $N$ by Lemma \ref{compress} so that for $n>8\|T\|$,
there exists a subfamily
$\{\mu_\beta: \beta \le \omega^n\}$ of
$\{T^*S^* \delta_\gamma: \gamma \le \omega^N\}$ such that $\beta \rightarrow
\mu_\beta$ is a (order to weak${}^*$) homeomorphism, $\beta \rightarrow
\gamma(\beta),$  defined by
$$\mu_\beta=T^*S^* \delta_{\gamma(\beta)},$$
is an order isomorphism and homeomorphism, $n \ge 8\|T\|/\epsilon$ and
$$|\|\mu_\beta\|-\|\mu_{\beta'}\||<\epsilon/(32\|T\|),$$
for all $\beta, \beta' \le
\omega^n.$
The family of measures $\{\delta_{\gamma(\beta)}:\beta \le \omega^n\}$
is the natural basis of the dual of a 1-complemented subspace $Z$ of
$C(\omega^\omega)$ isometric to $C(\omega^n)$. Indeed,  according to Remark \ref{pro} it suffices to take $Z$
as the subspace  of $C(\omega^\omega)$ of all
functions constant on order 
intervals $(\gamma(\beta),\gamma(\beta+1)]$ for $\beta < \omega^n.$ 
Further because $$\lim_K \|\mu_{\omega^n} |_{[K,\omega) \times \beta \mathbb
N 
} \| = 0,$$ 
and the restriction to $[1,K]\times \beta \mathbb N $ is
weak${}^*$-continuous,
we can assume that the support of $\mu_\gamma$ is contained in $[1,K]
\times \beta \mathbb N$ for all $\gamma \le \omega^n.$ Notice that
$[1,K]\times \beta \mathbb N $ is homeomorphic to $\beta \mathbb N $ so we
may replace $[1,K]\times \beta \mathbb N $ by $\beta \mathbb N $.

In order to simplify notation we can now assume that we have
a projection $T$ from
$C(\beta \mathbb N, X)$ onto a subspace $Y$
isomorphic by an operator $S$ to
$C(\omega^n)$ such that
$$\|T^*S^*z\| \ge \epsilon \|z\|,$$
for all $z \in
C(\omega^n)^*$, and
$$|\|T^*S^* \delta_\beta\|-\|T^*S^* \delta_{\beta'}\||<\epsilon/(8\|T\|),$$
for all $\beta, \beta' \le
\omega^n.$
Let
$$g_{\omega^n} = S^{-1}(1_{(0,\omega^n]}) \ \ \hbox{and} \ \ g_{\omega^{n-1}k}=S^{-1}(1_{(\omega^{n-1}(k-1),\omega^{n-1}k]}),$$
for all $k \in \mathbb N.$ $(g_{\omega^{n-1}k})$  is w.u.c.
Because $X$ does not contain $c_0$, for each $t$, $(g_{\omega^{n-1}k}(t))$ is
unconditionally converging and thus
converges in norm to 0 for all $t \in \beta \mathbb N.$ Because
$\|g_{\omega^{n-1}k}(\cdot)\|\le \|S^{-1}\|$ for all $k$,
$(\|g_{\omega^{n-1}k}(\cdot)\|)$
converges to 0 weakly in $C(\beta \mathbb N).$ Because $$\int g_{\omega^{n-1}k} \,
d\mu_{\omega^{n-1}k} = 1,$$  by Lemma \ref{disjoint_open}
there exists a subsequence
$(\mu_{\omega^{n-1}k})_{k\in M_1}$ and a sequence of disjoint clopen sets
$(G_{\omega^{n-1}k})_{k\in M_1})$ such that
$$ \int g_{\omega^{n-1}k} 1_{G_{\omega^{n-1}k}} \,
d\mu_{\omega^{n-1}k} \ge 7/8,$$
for all $k \in M_1.$ If there is an infinite
subset $K$ of $M_1$ and $\delta>0$
such that \begin{equation} \left|\int g_{\omega^n} 1_{G_{\omega^{n-1}k}} \,
d\mu_{\omega^{n-1}k}\right | \ge \delta,\label{delta}\end{equation}
for all $k \in K$, then Lemma \ref{linfty} would imply
that $C(\omega^\omega)$ is non-separable. Notice that the same contradiction would
result if for each $k$, we replace $G_{\omega^{n-1}k}$ in (\ref{delta}) by any of its clopen
subsets.

We also have that $$\int g_{\omega^n} \,
d\mu_{\omega^{n-1}k} =1,$$ for all $k$, thus, by replacing $M_1$ by an infinite
subset,
for each $k \in M_1$ there are disjoint
clopen sets $G^0_{\omega^{n-1}k}$ and $G^1_{\omega^{n-1}k}$ such that
\begin{itemize}
\item $\int g_{\omega^{n-1}k} 1_{G^1_{\omega^{n-1}k}} \
d\mu_{\omega^{n-1}k} >  3/4,$
\item $\int g_{\omega^n} 1_G \,d\mu_{\omega^{n-1}k} < 1/8$ for all clopen
$G \subset
G^1_{\omega^{n-1}k},$
\item $\int g_{\omega^n} 1_{G^0_{\omega^{n-1}k}}\,
d\mu_{\omega^{n-1}k}  > 3/4.$
\end{itemize}
This is the first step of an at most $n$-step induction argument.

Fix $k_1 \in M_1$. Consider the sequence  $(\omega^{n-1}(k_1-1)+\omega^{n-2}k)$.
For sufficiently large $k$,
$$\int g_{\omega^{n-1}k_1} 1_{G^1_{\omega^{n-1}k_1}} \,
d\mu_{\omega^{n-1}(k_1-1)+\omega^{n-2}k} >  3/4,$$
and
$$\int g_{\omega^n} 1_{G^0_{\omega^{n-1}k_1}} \,
d\mu_{\omega^{n-1}(k_1-1)+\omega^{n-2}k} >3/4.$$
Because $$g_{\omega^{n-1}(k_1-1)+\omega^{n-2}k}=S^{-1}(1_{(\omega^{n-1}(k_1-1)+\omega^{n-2}(k-1),\omega^{n-1}(k_1-1)+\omega^{n-2}k]})$$
converges weakly to 0, by applying Lemma \ref{disjoint_open}, there exist a subsequence $$(\mu_{\omega^{n-1}(k_1-1)+\omega^{n-2}k})_{k\in M_2}$$
and disjoint clopen sets $(G_k)_{k\in M_2}$ such that
$$\int g_{\omega^{n-1}(k_1-1)+\omega^{n-2}k} \,
d\mu_{\omega^{n-1}(k_1-1)+\omega^{n-2}k}>7/8,$$ for all $k\in M_2$.
For every $\delta>0$ and clopen $H_k \subset G_k$ for $k\in M_2$, Lemma \ref{linfty}
tells us that there are only finitely
many k for which $$|\int g_{\omega^{n-1}k_1} 1_{H_k} \,
d\mu_{\omega^{n-1}(k_1-1)+\omega^{n-2}k}| >\delta,$$
or
$$|\int
g_{\omega^{n}} 1_{H_k} \,
d\mu_{\omega^{n-1}(k_1-1)+\omega^{n-2}k}| >\delta.$$
Thus taking $\delta=1/8,$
for sufficiently large $k\in M_2$ we can find disjoint clopen sets
$$G^j_{\omega^{n-1}(k_1-1)+\omega^{n-2}k},$$
$j=0, 1, 2,$ such that
\begin{itemize}
\item  $\int g_{\omega^{n-1}(k_1-1)+\omega^{n-2}k} 1_{G^2_{\omega^{n-1}(k_1-1)+\omega^{n-2}k}}
\
d\mu_{\omega^{n-1}(k_1-1)+\omega^{n-2}k} >  5/8,$
\item  $\int g_{\omega^{n-1}k_1} 1_{G^1_{\omega^{n-1}(k_1-1)+\omega^{n-2}k}} \,
d\mu_{\omega^{n-1}(k_1-1)+\omega^{n-2}k} >  5/8,$
\item $\int g_{\omega^n} 1_{G^0_{\omega^{n-1}(k_1-1)+\omega^{n-2}k}} \,
d\mu_{\omega^{n-1}(k_1-1)+\omega^{n-2}k} >5/8.$
\end{itemize}
An induction argument shows that we can choose $k_1, k_2, \dots ,k_n$ and
disjoint clopen sets $$G^j_{\omega^{n-1}(k_1-1)+\omega^{n-2}(k_2-1)+\dots
k_n},$$ $j= 0, 1, \dots , n-1,$ such that
$$\int g_{\omega^{n-1}(k_1-1)+\omega^{n-2}(k_2-1)+\dots+\omega^{n-j}k_{j}}
1_{G^j_{\omega^{n-1}(k_1-1)+\dots
k_n}}
\,
d\mu_{\omega^{n-1}(k_1-1)+\dots +
k_n}$$
is strictly greater than $1/2+1/2^n$. This implies that
$$\|\mu_{\omega^{n-1}(k_1-1)+\omega^{n-2}(k_2-1)+\dots
k_n}\|>n/2>\|T\|\|\delta_{\omega^{n-1}(k_1-1)+\omega^{n-2}(k_2-1)+\dots
+k_n}\|.$$ This contradiction shows that no such projection $T$ exists.
\end{proof}

\begin{remark} The conclusion of this proposition is equivalent to the
statement that $C(\beta \mathbb N \times \omega,X)$ does not contain
$C(\omega^n)$ uniformly complemented. Obviously if $X$ contains $C(\omega^n)$
uniformly complemented then it follows that $C(\beta \mathbb N \times
\omega,X)$ contains $C(\omega^n)$
uniformly complemented. It is possible that the hypothesis on $X$ could be
weakened to something like $X$ does not contain $C(\omega^n)$ uniformly or
uniformly complemented.
We do not know whether $C(\beta \mathbb N \times \omega,C(\beta \mathbb N ))$
contains a complemented subspace isomorphic to $C(\omega^\omega)$. If this
is a counterexample then assuming additionally that $X$ is separable may
provide a strong enough hypothesis.
\end{remark}

The next result generalizes the previous to larger ordinals.

\begin{theorem}\label{alphacom} Let $X$ be a Banach space and $0 \leq \alpha < \beta< \omega_{1}$. Then
$$C(\omega^{\omega^{\beta}})  \st{c}{\hr} C(\beta \mathbb N \times \omega^{\omega^{\alpha}}, X) \Longrightarrow c_{0} \hr X.$$
\end{theorem}

\begin{proof} Let $\alpha$ be a countable ordinal
and $X$ a Banach space  containing no copy of $c_0$. 
We will show by induction that $\alpha$ is the smallest
ordinal $\gamma$ such that $C(\beta \mathbb N \times 
\omega^{\omega^\gamma}, X)$ contains a complemented subspace isomorphic to
$C(\omega^
{\omega^\alpha}).$ Theorem \ref{com} shows that for $\alpha = 1$ this is
the case.
Assume that the result holds for ordinals less than $\alpha$, $\alpha>1$, and that
$\gamma<\alpha$ is the smallest ordinal such that $C(\beta \mathbb N \times 
\omega^{\omega^\gamma}, X)$ contains a complemented subspace isomorphic to
$C(\omega^ {\omega^\alpha}).$ We will show that this leads to a
contradiction.

In place of $C(\beta \mathbb N \times 
\omega^{\omega^\gamma}, X)$, we will use the isomorphic space,
$C_0(\omega^{\omega^\gamma}\times \beta
\mathbb N, X).$ Now assume that $T$ is a
projection defined on $C_0(\omega^{\omega^\gamma}\times \beta \mathbb N, X)$
with range isomorphic to $C(\omega^ {\omega^\alpha}).$ Let $\alpha_k
\uparrow \omega^\alpha$ and $\beta_k \uparrow \omega^\gamma$, where $\alpha_k = \omega^{\alpha'}k$ if $\alpha=\alpha'+1$
for some $\alpha'$, or $\alpha_k=\omega^{\xi_k}$ if $\alpha$ is a limit ordinal
and $\xi_k \uparrow \alpha,$ and
$\beta_k = \omega^{\gamma'}k$ if $\gamma=\gamma'+1$
for some $\gamma'$, or $\beta_k=\omega^{\gamma_k}$ if $\gamma$ is a limit
ordinal
and $\gamma_k \uparrow \gamma.$
Choose $k_0$ such
that $\alpha_k\ge \omega^\gamma$ for all $k \ge k_0.$ By  Lemma \ref{compress} for each
$k>k_0$, there is a $k'$ such that $\{T^*\delta_\beta: \beta \le
\omega^{\alpha_{k'}}\} $ contains a subfamily $\{\mu_\rho: \rho \le
\omega^{\alpha_k}\}$ such that $\rho \rightarrow \mu_\rho$ is a homeomorphism,
$\rho \rightarrow \beta(\rho)$ is an order homeomorphism, where
$\mu_\rho=T^*\delta_{\beta(\rho)}$, and $$|\|\mu_\rho\| -\|\mu_{\rho'}\||<1/
(4\|T\|),$$ for all $\{\rho,\rho' \le \omega^{\alpha_k}\}$. For each $m$
let $P_m$ be the
canonical projection from 
$C_0(\omega^{\omega^\gamma} \times \beta
\mathbb N, X)$ onto $C(\omega^{\beta_m} \times \beta
\mathbb N, X).$ There exists
an $m$ such that $$\|(I-P_m^*)(\mu_{\omega^{\alpha_k}})\|<1/(8\|T\|).$$
It follows
by passing to a suitable neighborhood of $\omega^{\alpha_k}$ that
we may assume that $$\|(I-P_m^*)(\mu_{\rho})\|<3/(8\|T\|),$$ for all  $\rho \le
\omega^{\alpha_k}$. According to Remark \ref{pro}  we can find a
1-complemented subspace $Z$ of $C(\omega^ {\omega^\alpha})$ which is isometric to
$C(\omega^ {\alpha_k})$ and has natural basis of its dual
$\{\mu_\rho: \rho \le
\omega^{\alpha_k}\}$. This implies that $P_m(Z)$ is a complemented  subspace of
$C(\beta \mathbb N \times \omega^{\beta_m}, X)$ isomorphic to
$C(\omega^ {\alpha_k})$. Because
$\beta_m<\omega^{\gamma}$ and $\alpha_k \ge \omega^{\gamma}$, $C(\beta
\mathbb N \times \omega^{\beta_m}, X)$ cannot contain a complemented copy of
$C(\omega^{\omega^{\gamma}})$ by the inductive
hypothesis. Thus we have a contradiction and the theorem is proved.
\end{proof}

Now we can prove

\begin{theorem}\label{excom} Let $X$ be a Banach space  containing no copy of
$c_{0}$, $K$ an infinite compact metric space and $0 \leq \alpha<
\omega_{1}$. Then
$$C(K)   \st{c}{\hr}     C(\beta {\mathbb N}  \times 
\omega^{\omega^\alpha}, X)    \Longleftrightarrow  C(K) \sim
C(\omega^{\omega^\xi}) \ \hbox{for some} \ \ 0 \leq \xi \leq \alpha.$$
\end{theorem}
\begin{proof}
Since that $C([0, 1])$ contains  complemented  copies of every
$C(\omega^{\omega^\alpha})$, $0\leq \alpha < \omega_{1}$,   
it follows directly from  Milutin's theorem and Theorem \ref{alphacom} that
$K$ must be countable if $C(K)   \st{c}{\hr}     C(\beta {\mathbb N}  \times 
\omega^{\omega^\alpha}, X)$. If $K$ is countable, $C(K)$ is isomorphic to
$C(\omega^{\omega^\xi})$ for some countable ordinal $\xi$ and Theorem
\ref{alphacom} determines the possible values of $\xi$. The converse is
obvious.
\end{proof}

\begin{remark} Because $C(\beta \mathbb N,l_1)$ is isomorphic to its
$c_0$-sum, see Theorem \ref{iso}, the  above result in the case $X=l_{1}$ does not mimic that
for the scalar case where there is an additional isomorphism class.
Indeed, since
$c_0$ is not isomorphic to a complemented subspace of $C(\beta \mathbb N)$,
$C(\beta \mathbb
N)$ and $C(\beta \mathbb
N\times \omega)$ are not isomorphic.
\end{remark}

We can now prove the main result of this section.

\begin{theorem}\label{main} Let $X$ be a Banach space containing no copy of
$c_{0}$. Then for any infinite compact metric spaces $K_{1}$ and $K_{2}$
we have
$$C(\beta {\mathbb N} \times K_{1}, X) \sim C(\beta {\mathbb N} \times
K_{2}, X) \Longleftrightarrow C(K_{1}) \sim C(K_{2}).$$
\end{theorem}

\begin{proof} Let us show the non trivial implication. Suppose that
$$C(\beta {\mathbb N} \times K_{1}, X) \sim C(\beta {\mathbb N} \times K_{2}, X).$$

It is convenient to consider two subcases:

Case 1. $K_{1}$ and $K_{2}$ are countable. Hence there are countable ordinals
$\xi$ and $\eta$ such that $C(K_{1}) \sim C(\omega^{\omega^\xi})$ and $C(K_{2})
\sim C(\omega^{\omega^\eta})$. Then according to our hypothesis,
$$ C(\omega^{\omega^\eta}) \st{c}{\hr} C(\omega^{\omega^\eta}\times \beta
{\mathbb N}, X) \sim C(\omega^{\omega^\xi} \times \beta \mathbb N,
X).$$
Therefore by  Theorem \ref{excom}. we deduce that
$\omega^{\omega^\eta} \leq \omega^{\omega^\xi}$. Similarly, we show that
$\omega^{\omega^\xi} \leq \omega^{\omega^\eta}$. Hence $C(K_{1}) \sim C(K_{2})$.

Case 2. $K_{1}$ or $K_{2}$ is uncountable. Without loss of generality we
suppose that $K_{2}$ is uncountable. To prove  that $C(K_{1}) \sim
C(K_{2})$, it is enough  by Milutin's theorem to show that $K_{1}$ is also
uncountable. Assume to the contrary. So there exists an ordinal $\xi$ such
that $C(K_{1}) \sim C(\omega^{\omega^\xi})$. Since $C(K_{2}) \sim C([0, 1])$ we
have by our hypothesis that
$$C(\omega^{\omega^{\xi +1}}) \st{c}{\hr} C(\beta \mathbb N \times [0, 1], X)
\st{c}{\hr}  C(\omega^{\omega^\xi} \times \beta {\mathbb N}, X),$$
a contradiction of Theorem \ref{excom}. This completes the proof of Theorem \ref{main}.

\end{proof}

\begin{corollary}
Let $X$ be a Banach space containing no copy of
$c_{0}$. Then for any infinite compact metric spaces $K_{1}$ and $K_{2}$
we have
$$C(K_{1}, X) \sim C(K_{2}, X) \Longleftrightarrow C(K_{1}) \sim C(K_{2})$$
\end{corollary}
\begin{proof} One direction is immediate.
If $C(K_{1}, X) \sim C(K_{2}, X)$ then $C(\beta {\mathbb N} \times K_{1},
X) \sim C(\beta {\mathbb N} \times
K_{2}, X)$, so this follows from the previous result.
\end{proof}
The next result can be considered as an extension of the Cembranos-Freniche
result however the proof does not yield a proof of that result. To include
the original we would need to use the Josefson-Nissenzweig Theorem, \cite{J}
and \cite{N}.

\begin{proposition}\label{extCF}
Suppose that $0\le \alpha <\omega_1$, $0\le \gamma <\omega_1$,
and $(\gamma_n)$ is either
$(\omega^{\beta_n})$ where $(\beta_n)$ increases to $\gamma$ or
$(\gamma_n)$ is $(\omega^\beta n)$ and $\gamma =
\beta +1$ for some ordinal $\beta$. $X$ is a Banach space such
that with constants independent of $n$, $C(\omega^{\gamma_n})$
is
isomorphic to a complemented subspace of $X$. Then for any infinite
compact Hausdorff
space $K$,
$C(\omega^{\omega^\alpha}\times \omega^{\omega^\gamma})$ is
isomorphic to a complemented subspace of
$C(K \times \omega^{\omega^\alpha},X).$

If $X$ is also separable then $C(\omega^{\omega^\gamma})$ is
isomorphic to a complemented subspace of
$C(K ,X)$.
\end{proposition}

\begin{proof} Clearly $C(\omega^{\omega^\alpha})$ is isomorphic to a
complemented subspace of $C(K \times \omega^{\omega^\alpha},X).$
We also know by Lemma \ref{ordprod} that $$C(\omega^{\omega^\alpha}\times
\omega^{\omega^\gamma}) \sim C(\omega^{\omega^{\max
\{\alpha,\gamma\}}}).$$ So we need only show that
$C(\omega^{\omega^\gamma})$ is isomorphic to a complemented subspace of 
$C(K\times \omega^{\omega^\alpha},X)$. The case $\gamma=0$ is the
Cembranos-Freniche result but also is immediate from the fact that $\alpha
\ge 0.$ Now assume $\gamma \ge 1.$

Notice that $C(K\times \omega^{\omega^\alpha}, X)$ is isomorphic to
$C_0(\omega \times \omega^{\omega^\alpha} \times K, X).$ This in turn
is isomorphic to 
$$(\sum_{j\in \mathbb N} C(\omega^{\omega^\alpha} \times K, X))_{c_0}.$$
For each $n \in \mathbb N$ let $X_n$ be a complemented subspace of $X$
which is isomorphic to $C(\omega^{\gamma_n})$ and let
$P_n$ be a projection from $X$ onto $X_n$. By the hypothesis we can assume
that the norms of the isomorphisms and the projections are bounded
independent of $n$. Choose any point $a \in
\omega^{\omega^\alpha} \times K$. If $(f_j) \in (\sum_j
C(\omega^{\omega^\alpha} \times K, X))_{c_0}$, then
$$P((f_j))=(P_j(f_j(a))1_{\omega^{\omega^\alpha} \times K})$$ defines a
projection onto a space isometric to the $c_0$-sum of $X_j$. Because
$(\sum_{j\in \mathbb N} C(\omega^{\gamma_j}))_{c_0}$ is isomorphic to
$C(\omega^{\omega^\gamma})$, the $c_0$-sum
of $X_j$ is isomorphic to $C(\omega^{\omega^\gamma})$.

If $X$ is separable, then with $X_n$ as before let $Y_n$ be the subspace of
$X_n$ which is the image of $C_0(\omega^{\gamma_n})$ under the isomorphism
from $C(\omega^{\gamma_n})$ and $Q_n$ be the
projection from $X$ onto $Y_n$. Because $X$ is separable there is a decreasing
sequence of 
weak${}^*$-open sets $G_j$ which is a base for the neighborhood system of $0$
in the ball of $X^*$, $B_{X^*}.$ For each $n$ there is a sequence of complemented
subspaces $Y_{n,k}$ of $Y_n$ with projections $Q_{n,k}$, $Y_{n,k}\supset
Y_{n,k+1}$ and $Y_{n,k}$ is isomorphic to $C_0(\omega^{\gamma_n})$ for all $k$, 
such that for each $j$ and $n$ there is a $K$
such that for all $k \ge K$, $$Q_{n,k}^*(X^*) \cap B_{X^*} \subset G_j.$$
Indeed if $\rho_k \nearrow \omega^{\gamma_n}$, then $f \rightarrow f 
1_{(\rho_k,\omega^{\gamma_n}]}$ is a projection onto a subspace of 
$C_0(\omega^{\gamma_n})$ isomorphic to $C_0(\omega^{\gamma_n})$ and $Y_{n,k}$
can be taken to be the image of this subspace in $Y_n$.

Let $(g_n)$ be a sequence of disjointly supported non-negative
norm one elements in $C(K)$ such that for each $n$ there is an open set $$H_n
\supset \overline{\{t:g_n(t)>0\}},$$ with $H_n\cap H_m = \emptyset$ for all $m
\ne n,$  and $t_n \in K$ such that $g_n(t_n)=1$ for all
$n$. Let $$D=\sup_{s,k} \|Q_{s,k}\|,$$ and choose $k_n$ such that
$$D^{-1}Q_{n,k_n}^*(B_{X^*}) \subset G_n,$$ for
all $n$ and let $$Z=[g_n y_n: y_n \in Y_{n,k_n}, n\in \mathbb N].$$ Clearly $Z$
is isomorphic to $(\sum_n C_0(\omega^{\gamma_n}))_{c_0}$ which is isomorphic to
$C(\omega^{\omega^\gamma})$. Define an operator $T$ from $C(K,X)$ into $Z$ by
by $$Tf(t) = g_n(t) Q_{n,k_n}(f(t_n)),$$ for all $t \in \text{ supp }g_n$ and
$Tf(t)=0$ if $t \notin \cup_n \text{ supp }g_n$. Because each $g_n$ is
continuous, $Tf$ is continuous on $H_n$ for all $n$. If $\epsilon>0$,
$t'_n \in \text{ supp
}g_n$, and $t$ is an accumulation point of $\{t'_n \}$, 
by the continuity of $f,$ $\|f(t)-f(t'_n)\|<\epsilon$ for all $t'_n
\in H$ where $H$ is some neighborhood of $t$. Because $t$ cannot be in any
$H_n$, $T(f(t))=0.$ By the choice of $k_n$ we have that
$$\lim_n \sup_{x^*\in B_{X^*}}|(Q_{n,k_n}^* x^*)(f(t))|=0$$ and for $t'_n \in H$,
$$\epsilon \sup_{s,k \in \mathbb N}\|Q_{s,k}\|
>\sup_{x^*\in B_{X^*}}|(Q_{n,k_n}^*x^*)(f(t'_n))-(Q_{n,k_n}^*
x^*)(f(t))|.$$
Thus 
$$\|(Tf)(t)\|=0=
\lim_{n\in \mathcal N} \sup_{x^*\in B_{X^*}}|(Q_{n,k_n}^*
x^*)(f(t_n))g(t'_n)|=\lim_{n\in \mathcal N}\|(Tf)(t'_n)\|,$$
where the
limit is over some net $(t'_n)_{n \in \mathcal N}$ so that
$\lim_{n\in \mathcal N}t'_n =t.$ 

It is easy to see that $$\|T\| \le \sup_{s,k \in \mathbb N}\|Q_{s,k}\|$$ and,
because $g_n(t_n)=1$ and each $Q_{n,k_n}$ is a projection, that $T$ is a
projection.
\end{proof}

\begin{remark}
We do not whether the separability condition in the second part
is necessary but the argument
fails for the natural choices of $X_n$ if $X=\{(x_n): \ \forall  \ n \in
\mathbb N,x_n
\in C(\omega^n),\|(x_n)\|=\sup_n \|x_n\|<\infty \}$.
\end{remark}

In the next section we will prove some results about quotients of $C(K,X)$
isomorphic to $C(\omega^{\omega^\alpha})$.
If we consider quotients in the previous proposition
instead of complemented subspaces, the
analogous results hold. The proof is similar to the previous one except
that the argument is now entirely in the dual.

\begin{proposition}\label{qextCF}
Suppose that $0\le \alpha <\omega_1$, $0\le \gamma <\omega_1$,
and $(\gamma_n)$ is either
$(\omega^{\beta_n})$ where $(\beta_n)$ increases to $\gamma$ or
$(\omega^\beta n)$ and $\gamma =
\beta +1$. $X$ is a Banach space such
that with constants independent of $n$, $C(\omega^{\gamma_n})$
is
isomorphic to a quotient of $X$. Then for any infinite
compact Hausdorff
space $K,$
$C(\omega^{\omega^\alpha}\times \omega^{\omega^\gamma})$ is
isomorphic to a quotient of
$C(K \times \omega^{\omega^\alpha},X).$

If $X$ is also separable then $C(\omega^{\omega^\gamma})$ is
isomorphic to a quotient of
$C(K ,X)$.
\end{proposition}

\begin{proof} Clearly $C(\omega^{\omega^\alpha})$ is isomorphic to a
quotient of $C(K \times \omega^{\omega^\alpha},X).$
We also know by Lemma \ref{ordprod} that $$C(\omega^{\omega^\alpha}\times
\omega^{\omega^\gamma}) \sim C(\omega^{\omega^{\max
\{\alpha,\gamma\}}}).$$ Thus as before we need only show that
$C(\omega^{\omega^\gamma})$ is isomorphic to a quotient of 
$C(K\times \omega^{\omega^\alpha},X)$. The case $\gamma=0$ is the
Cembranos-Freniche result but also is immediate from the fact that $\alpha
\ge 0.$ Assume $\gamma \ge 1.$

Notice that $C(K\times \omega^{\omega^\alpha}, X)$
is isomorphic to
$C_0(\omega \times \omega^{\omega^\alpha} \times K, X)$ and this is isomorphic to 
$$(\sum_j C(\omega^{\omega^\alpha} \times K, X))_{c_0}.$$
For each $n \in \mathbb N$ let $X_n$ be a quotient of $X$
which is isomorphic to $C(\omega^{\gamma_n})$. 
$P_n$ be a quotient map
from $X$ onto $X_n$. By the hypothesis we can assume
that the norms of the isomorphisms and the quotient maps are bounded
independent of $n$. Thus $P_n^*(X_n^*)$ is weak${}^*$-isomorphic to
$C(\omega^{\gamma_n})^*$ and the subspace $$Z=\{(z_n): z_n \in
P_n^*(X_n^*)\text{ for all } n\in \mathbb N\}$$
 of $(\sum_j
C(\omega^{\omega^\alpha} \times K, X))_{c_0})^*$ is weak${}^*$-isomorphic to 
$((\sum_j
C(\omega^{\gamma_n}))_{c_0})^*$. This space is weak${}^*$-isomorphic to
$C(\omega^{\gamma})^*$ giving us the required quotient.

If $X$ is separable, then with $X_n$ as before let $Y_n$ be the
complemented subspace of
$X_n$ which is the image of $C_0(\omega^{\gamma_n})$ under the isomorphism
and $Q_n$ be the
the quotient map from $X$ onto $Y_n$. Because $X$ is separable there is a decreasing
sequence of 
weak${}^*$-open sets $G_j$ which is a base for the neighborhood system of $0$
in the ball of $X^*$, $B_{X^*}.$ As in the proof of Proposition \ref{extCF}
for each $n$ there is a sequence of complemented
subspaces $Y_{n,k}$ of $Y_n$ with projections $Q_{n,k}$, $Y_{n,k}\supset
Y_{n,k+1}$ and $Y_{n,k}$ is isomorphic to $C_0(\omega^{\gamma_n})$ for all $k$, 
such that for each $j$ and $n$ there is a $K$
such that for all $k \ge K$, $$Q_{n,k}^*(Y_n^*) \cap B_{X^*} \subset G_j.$$

Let $(t_n)$ be a sequence of  points in $K$
such that for each $n$ there is an open set $H_n$ containing $t_n$
with $H_n\cap H_m = \emptyset$ for all $m
\ne n,$ for all 
$n$. Let $$D=\sup_{s,k} \|Q_{s,k}\|,$$ and choose $k_n$ such that
$$D^{-1}Q_{n,k_n}^*(B_{Y_n^*}) \subset G_n,$$ for
all $n$ and let $$Z=[ z_n \delta_{t_n}: z_n \in Q_{n,k_n}^*(Y_{n}^*)\text{
for all } n\in \mathbb N].$$ Clearly $Z$
is isomorphic to $((\sum_n C_0(\omega^{\gamma_n})^*)_{\ell_1}$. We  need to
show that Z is w${}^*$-isomorphic to $((\sum_n
C_0(\omega^{\gamma_n}))_{c_0})^*$. This however follows immediately from
the choice of $(k_n)$ and the fact that no point $t_j$ 
is an accumulation point of $\{t_n: n \in \mathbb N\}$. (As was shown in
the proof of Proposition \ref{extCF}).
\end{proof}

\section{Separable $C(K)$ quotients
of $C(\beta \mathbb N\times \alpha, X)$ } \label{quotients}

By Theorem \ref{CF}
we know that $C(\beta
\mathbb N, l_{p})$,  $1< p< \infty$, contains a complemented copy of $c_{0}$.
The main aim of this section is to show that
$C(\omega^\omega)$ is not even a quotient of this space,
(Proposition \ref{omomX}). Of course, this implies that $c_{0}$ is, up to an isomorphism, the only separable $C(K)$ space which is a quotient of $C(\beta
\mathbb N, l_{p})$, $1< p< \infty$.

In this section we will work with Banach spaces $X$ that satisfy the
following properties.
\begin{enumerate}
\item[($\dagger$)] $X^*$ has a monotone weak${}^*$-FDD $(X^*_m)$.
\item[($\ddagger$)] For every constant $C$, $0<C<1$,
there is a constant $C'$ such that
for all $x^*\in X^*$ and $j\in \mathbb N$,
$$\|(I-P_j)x^*\|\le C\|x^*\|+C'(\|x^*\|-\|P_jx^*\|),$$
where $(P_j)$ is the sequence of FDD projections, i.e.,
$P_j(X^*)=[X^*_m:m\le j]$.
\end{enumerate}
We will refer to such spaces as {\it satisfying  the daggers.}
Before proceeding to the main results we will verify that $\ell_p$, for
$1<p<\infty$, satisfies the daggers.

The following lemma follows from the Mean Value Theorem.

\begin{lemma} \label{meanvalue} Suppose that $0<C<1$, $1<q<\infty$,
$g(t)=(1-t^q)^{1/q}$ and $t_0$ satisfies $g(t_0)=C$. Then there is a
positive constant $C_1$ which depends on $q$  such that for all $t_0>t>0$
$$g(t)\le C+C_1(t_0-t).$$
\end{lemma}

If we consider $\ell_p^*=\ell_q$ with the standard basis and basis
projections, then we have that for $x^*\in \ell_q$, $x^*\ne 0$,
$$\|x^*\|=(\|P_j
x^*\|^q+\|x^*-P_jx^*\|^q)^{1/q}.$$ Let $\lambda=\|P_j x^*\|/\|x^*\|$ and 
rewriting we have that $$\|x^*-P_jx^*\|=\|x^*\|(1-\lambda^q)^{1/q}.$$
If $C$ is given and $\lambda \le t_0$ where $t_0$ and $C_1$ are as in the
lemma, 
$$\|x^*-P_jx^*\|\le C\|x^*\|+C_1\|x^*\|(t_0-\lambda)\le
C\|x^*\|+C_1(\|x^*\|-\|P_j x^*\|).$$
If $\lambda>t_0$, then $(1-\lambda^q)^{1/q}<C$ and the first term suffices.
Thus we see that

\begin{corollary} $\ell_p$, $1<p<\infty$, satisfies the daggers.
\end{corollary}

The next result is the initial case of the main theorem of this section.
 
\begin{theorem}\label{omomX} Suppose that $X$ is a Banach space
satisfying the daggers. Then $C(\omega^\omega)$
is not a quotient of  $C(\omega \times \beta {\mathbb N}, X).$
\end{theorem}

\begin{proof} First we can replace $C(\omega \times \beta {\mathbb N}, X)$
by the isomorphic space $C_0(\omega \times \beta {\mathbb N}, X).$
Suppose  that there exists
a bounded linear operator $T$ from $C_0(\omega \times \beta {\mathbb N}, X)$
onto $C(\omega^\omega)$.  We will show that this  leads to a contradiction. We
may assume that $\|T\|=1$.  Because $T^*$ is a weak${}^*$-isomorphism
from $C(\omega^\omega)^*$ into $C_0(\omega \times \beta\mathbb N, X)^*$,
there is a constant $K$ such that for every $n \in \mathbb N$,
$\{\delta_\alpha:\alpha \le \omega^n\}$ is mapped to $\{x^*_\alpha:\alpha \le
\omega^n\}\subset C_0(\omega \times \beta\mathbb N, X)^*$
and $\{x^*_\alpha:\alpha \le
\omega^n\}$ is $K$-equivalent to the the usual unit vector basis of
$\ell_1$, i.e., $$\|\sum_\alpha c_\alpha x^*_\alpha\|\ge \sum_\alpha
|c_\alpha|/K,$$
for all sequences of scalars $(c_\alpha)$.

Let $C=(8K)^{-1}$ in $(\ddagger)$, and choose $\rho,$
$1/2>\rho>0$, such that $$\left
((1-\rho)^{-1}(1-\frac{(1-\rho)^3}{1+\rho}\right){C'}<(4K)^{-1}.$$

By Lemma \ref{compress} for $k=1$ and $n$ sufficiently large,
we can find $(x^*_{\alpha(\gamma)})_{\gamma \le \omega}$
with
$$|\|x^*_{\alpha(\gamma)}\|
-\|x^*_{\alpha(\gamma')}\| |<\rho/K,$$
for all $\gamma,\gamma'\le \omega$. Moreover, as in the proof of Theorem
\ref{com}, we may assume that the measures are all supported in
$[1,K]\times \beta \mathbb N$ and thus reduce to measures supported on
$\beta \mathbb N.$ By
our identification of $C(\beta\mathbb N,X)^*$ with a space of
$X^*$-valued measures and switching
to a more suggestive notation we have $(\mu_n)$, a weak${}^*$ converging
sequence of $X^*$-valued measures with limit $\mu$, with
$$\|\mu\|(1+\rho)> \|\mu_n\|>\|\mu\|(1-\rho),$$
for all $n$. 

Choose a finite partition $\{B_j\}$
of $\beta \mathbb N$ into clopen sets
such that $$\sum_j \|\mu(B_j)\|_{X^*} >\|\mu\|(1-\rho).$$
As in $(\ddagger)$ let $P_m$ denote the
FDD projection of $X^*$ onto the span of the
first $m$ subspaces. In a slight
abuse of notation we will also use $P_m$ for the operator on the
$X^*$-valued measures defined by
$$(P_m \mu)(A)=P_m(\mu(A)),$$
for all measurable $A$. Choose $N$
such that $$\|P_N(\mu(B_j))\|>(1-\rho)\|\mu(B_j)\|,$$
for all $j$. By passing
to a subsequence we may assume that $$\|P_N(\mu_n(B_j))\|>(1-\rho)\|\mu(B_j)\|,$$
for all $j$ and $n$. Hence
\begin{align*}
\|P_N \mu_n\| \ge \sum_j \|P_N\mu_n(B_j)\|&\ge (1-\rho)\sum_j
\|\mu(B_j)\| \ \\ &\ge (1-\rho)^2 \|\mu\| \ge
\frac{(1-\rho)^2}{1+\rho}\|\mu_n\|. \end{align*}
Fix $n$ and choose a partition $\{A_k\}$ that refines $\{B_j\}$, such that
\begin{itemize}
\item $\sum_k \|\mu_n(A_k)\|\ge (1-\rho)\|\mu_n\|,$
\item $\sum_k \|P_N\mu_n(A_k)\|\ge (1-\rho)\|P_N\mu_n\|,$
\item $\sum_k
\|(I-P_N)\mu_n(A_k)\|\ge (1-\rho)\|(I-P_N)\mu_n\|.$
\end{itemize}
For each $k$
let $\lambda_k$ satisfy
$$\lambda_k \|\mu_n(A_k)\|=\|P_N\mu_n(A_k)\|.$$
Then
$$\sum_k \lambda_k \|\mu_n(A_k)\|\ge \frac{(1-\rho)^3}{1+\rho}\|\mu_n\|\ge
\frac{(1-\rho)^3}{1+\rho}\sum_k \|\mu_n(A_k)\|.$$
Equivalently,
$$(1-\frac{(1-\rho)^3}{1+\rho})\sum_k \|\mu_n(A_k)\| \ge \sum_k (1-\lambda_k)
\|\mu_n(A_k)\|.$$

We need to estimate $\|(I-P_N)\mu_n\|$.
Let
$$M=\{k:\|(I-P_k)\mu_n(A_k)\|\le C\|\mu_n(A_k)\|\}.$$
Then by the choice of $C$,  $\sum_k
\|(I-P_N)\mu_n(A_k)\|$ is less than or equal to
\begin{align*}
&\sum_{k\in M} C \|\mu_n(A_k)\|+\sum_{k\notin M}
(C+C'(1-\lambda_k))\|\mu_n(A_k)\| \\
&\le \sum_k C
\|\mu_n(A_k)\|+C'\sum_k(1-\lambda_k)\|\mu_n(A_k)\|\\
&\le C
\|\mu_n\|+C'(1-\frac{(1-\rho)^3}{1+\rho})\|\mu_n\|.\end{align*}
Because $\|T\|=1,$ $\|\mu_n\| \le 1.$
Therefore $\|(I-P_N)\mu_n\|$ is less than or equal to
$$(1-\rho)^{-1}C+(1-\rho)^{-1}C'(1-\frac{(1-\rho)^3}{1+\rho})
\le (2K)^{-1}.$$

Notice that $(P_N\mu_n)$ converges to $P_N \mu$ in the weak${}^*$-topology.
Because $C(\beta \mathbb N)$ is a Grothendieck space
so is $C(\beta \mathbb
N,[X_m:m\le N])$. Thus $(P_N\mu_n)$ converges to $P_N \mu$ in the weak
topology.  Because $(\mu_n)$ is $K$-equivalent to the usual unit vector basis of
$\ell_1$, the estimate on $\|(I-P_N)\mu_n\|$ implies that
$(P_N\mu_n)$ is also
equivalent to the usual unit vector basis of
$\ell_1$. This is a contradiction because the unit vector basis of
$\ell_1$ has no weak Cauchy subsequence.
\end{proof}



\begin{remark}The proof of the theorem shows that $C(\omega^n)$ is not a quotient  of
$C(\omega \times \beta \mathbb N, X)$ uniformly in $n$.
Because $C(\omega^\omega)$ is isomorphic to $(\sum_n C(\omega^n))_{c_0}$,
this is equivalent to $C(\omega^\omega)$ not
being a quotient.
It is conceivable that the conclusion of Theorem \ref{omomX}
could be proved under  the hypothesis that $X$ does not have $C(\omega^n)$
as a quotient uniformly in $n$.
\end{remark}

The following theorem 
generalizes Theorem \ref{omomX} to higher ordinals.

\begin{theorem}\label{omalX}  Suppose that $X$ is a Banach space such that
$C(\omega^\omega)$ is not isomorphic to a quotient of $C(\omega,X)$. Let
$\alpha \ge 1.$ Then $C(\omega^{\omega^\alpha})$ is isomorphic to a
quotient of $C(\beta \mathbb N \times
\xi,X)$ if and only if $\xi \ge \omega^{\omega^\alpha}.$
\end{theorem}

\begin{proof} It is easy to see that $C(\gamma)$ is
isomorphic to a complemented subspace of $C(\beta \mathbb N \times
\gamma,X)$ for all $\gamma \ge
1.$ Thus the sufficiency is clear.

On the other hand,  by Proposition \ref{BPiso}  $$C(\beta \mathbb N \times
\omega^n,X)=C(\omega^n,C(\beta \mathbb N,X)) \sim C(\omega,C(\beta \mathbb N,X))=C(\beta \mathbb N \times\omega,X),$$
for all positive
integers $n$. Thus the necessity is true for $\alpha =1.$ Because we have that $C(\omega^\omega)$ is not isomorphic to
a quotient of $C(\omega,X)$ or of $C(\beta \mathbb N \times \omega,X)$, it
is sufficient to prove the result for $C(\xi,X)$ rather than $C(\beta
\mathbb N \times
\xi,X)$. 

Now suppose that
$\alpha > 1$ and
for all $\gamma<\alpha,$ $C(\omega^{\omega^\gamma})$ is not isomorphic to a
quotient of $C(\xi,X)$ if $\xi<\omega^{\omega^\gamma}.$ 
We will show that if $\beta<\omega^{\omega^\alpha}$ and
there is a bounded operator $T$ from $C(\beta,X)$ onto
$C(\omega^{\omega^\alpha})$ that this leads to a contradiction. Without loss of generality
we assume that $\|T\|=1.$

If $\alpha$ is a limit ordinal and $\alpha_n
\uparrow \alpha,$ then $\omega^{\omega^{\alpha_n}}>\beta$ for some $n$.
$C(\omega^{\omega^{\alpha_n}})$ is isomorphic to a quotient of $C(\omega^{\omega^\alpha})$ and by the assumption also $C(\beta,X)$ but
this contradicts the induction hypothesis.

Now suppose that $\alpha=\gamma+1,$ for some ordinal $\gamma.$ By
Proposition \ref{BPiso},
$C(\beta,X)$ is
isomorphic to $C_0(\beta,X)$. Hence we may assume that $T$ goes from
$C_0(\beta,X)$ onto $C(\omega^{\omega^\alpha})$. Also by Proposition
\ref{BPiso}
we may assume that $\beta =
\omega^{\omega^\zeta}$ for some $1 \le\zeta <\alpha$. Let $\zeta_k =
\omega^{\zeta - 1}k$ if $\zeta$ is not a limit ordinal and $\zeta_k
\uparrow \omega^\zeta$ otherwise. Let $\{x^*_{\rho}:\rho\le
\omega^{\omega^\alpha}\}$ be the corresponding images of the dirac measures
$\{\delta_{\rho}:\rho\le
\omega^{\omega^\alpha}\}$ under $T^*$.
Then $\{x^*_{\rho}:\rho\le
\omega^{\omega^\alpha}\}$ is, for some $R>1$, $R$-equivalent to the usual unit vector basis
of $\ell_1$. Employing Lemma \ref{compress} there is an $N$ sufficiently large
such that if $\{y^*_\rho:\rho \le \omega^{\omega^\gamma \dot N}\}$ is
contained in the unit ball of $C_0(\beta,X)^*$ and the mapping $\rho \rightarrow
y^*_\rho$ is $\hbox{weak}^*$ continuous, then there is a continuous
map $\phi$ from
$[1,\omega^{\omega^\gamma}]$ into $[1,\omega^{\omega^\gamma \dot N}]$ such
that $$|\|y^*_{\phi(\rho)}\|-\|y^*_{\phi(\rho')}\||<\frac{1}{4R},$$ for all $\rho,\rho'\le
\omega^{\omega^\gamma}.$ Applying this to $\{x^*_{\rho}:\rho\le
\omega^{\omega^\gamma\dot N}\}$, we get a map $\phi$ as above.
Because $C_0(\beta,X)^*$ is
$\hbox{weak}^*$-isomorphic, norm-isometric
to $(\sum C(\omega^{\zeta_k},X)^*)_{c_0^*}$, there exists $K$ such
that $$\|(I-P_K) x^*_{\phi(\omega^{\omega^\gamma}}\|<\frac{1}{4R},$$
where
$P_K$ is the $\hbox{weak}^*$-continuous projection (truncation)
from $$(\sum_{k=1}^\infty
(C(\omega^{\zeta_k},X)^*)_{c_0^*}  \ \ \hbox{onto} \ \ (\sum_{k\le K}
(C(\omega^{\zeta_k},X)^*)_{c_0^*}.$$
By passing to a neighborhood  of $\omega^{\omega^\gamma}$ we may assume that
$$\|P_K x^*_{\phi(\rho)}\|>\|P_K x^*_{\phi(\omega^{\omega^\gamma})}\| -
\frac{1}{4R} > \|x^*_{\phi(\rho)}\|-\frac{3}{4R},$$ for all $\rho \le
\omega^{\omega^\gamma}.$ Thus $C(\omega^{\omega^\gamma})^*$ is
$\hbox{weak}^*$-isomorphic to a subspace of 
$$(\sum_{k\le K}
C(\omega^{\zeta_k},X)^*)_{c_0^*},$$ and consequently
$C(\omega^{\omega^\gamma})$ is isomorphic to a quotient of
$$C(\omega^{\zeta_1}+\dots+\omega^{\zeta_K},X).$$
But $\omega^{\zeta_1}+\dots+\omega^{\zeta_K}< \omega^{\omega^\gamma}$
contradicting the inductive hypothesis. No such $\beta$ exists.
\end{proof}

The purpose of this
section is to prove Theorem \ref{exquo}.  This result now follows from
Theorem \ref{omalX} and Milutin's theorem by an argument similar to the
deduction of Theorem \ref{excom} from Theorem \ref{com}. We leave the
details to the reader.

\begin{theorem}\label{exquo} Let $K$ be an infinite compact metric space and
$0 \leq \alpha < \omega_{1}$, and let $X$ satisfy the daggers. Then
$$C(\beta {\mathbb N} \times \omega^{\omega^\alpha}, X) \sai  C(K)
\Longleftrightarrow  C(K) \sim C(\omega^{\omega^\xi}) \ \hbox{for some} \ \
0 \leq \xi \leq \alpha.$$
\end{theorem}  

\begin{corollary}\label{exquolp} Let $K$ be an infinite compact metric space and
$0 \leq \alpha < \omega_{1}$. Then
$$C(\beta {\mathbb N} \times \omega^{\omega^\alpha}, l_{p}) \sai  C(K)
\Longleftrightarrow  C(K) \sim C(\omega^{\omega^\xi}) \ \hbox{for some} \ \
0 \leq \xi \leq \alpha.$$

\end{corollary}

\section{The isomorphism of $C(\beta \mathbb N \times \omega, l_{p})$  and  $C(\beta \mathbb N, l_{p})$,    $1 \leq p < \infty$}
\label{Clpom}

As an immediate consequence of Theorem \ref{com},
for every $1 \leq p <
\infty$ and $\alpha\ge \omega^{\omega}$, we have
$$C(\beta
\mathbb N \times \alpha, l_{p}) \not \sim C(\beta \mathbb N, l_{p}).$$
For $\alpha$ finite $\beta \mathbb N \times \alpha$ is homeomorphic to
$\beta \mathbb N$, so the case $\omega \le \alpha <\omega^\omega$ of
Theorem \ref{iso} remains.
This result is the case $p=1$ of Theorem \ref{isok}. To prove this
theorem  we need the following lemma.

\begin{lemma}\label{c0sum} Let $X$ and $Y$ be Banach spaces and let $1\leq
p  < \infty$ and $p \leq q\le \infty$ with $(p, q) \neq (1, \infty)$.  Then
$c_{0}(\omega \times {\mathcal K}
(X, Y))$ is isomorphic to a complemented subspace of ${\mathcal K} (l_{p}(X),
l_{q}(Y))$.
\end{lemma}

\begin{proof} Let $T\in{\mathcal K} (l_{p}(X),
l_{q}(Y))$. Represent $T$ as a matrix with entries in ${\mathcal K} (X,Y)$.
For the moment assume that $q<\infty.$
By \cite[Proposition
1.c.8 and following Remarks]{LT} the operator given by the diagonal of the matrix is a
bounded linear operator with norm no larger than $\|T\|$. Therefore the mapping
from ${\mathcal K} (l_{p}(X),
l_{q}(Y))$ into the diagonal operators  with respect to this representation is a
contraction.  If $p>1$ and $q=\infty$, then sign change operators are
contractive and the argument from \cite{LT} shows that the map from
${\mathcal K} (l_{p}(X),
l_{\infty}(Y))$ into
the diagonal operators is contractive. Also for compact operators that are
diagonal the map is the identity.

The norm of a diagonal operator is the supremum of the
norms of the operators on the diagonal. For the case $q=\infty$ this is
clear. If $q<\infty$, let $D_j$ be the $j$th block of
the diagonal operator $D$ and $(x_j) \in l_{p}(X)$. Then
\begin{align*}
\|(D_j x_j)\|_{l_{q}(Y)}=\left ( \sum \|D_j x_j\|_Y^q\right )^{1/q} &\leq\left ( \sum
\|D_j\|^q \|x_j\|_X^q\right )^{1/q} \\ &\le (\sup_j \|D_j\|) \|(x_j)\|_{l_p(X)},  \end{align*}    since $q
\ge p.$ Clearly
$\|D\| \ge \sup \|D_j\|.$

Let $E_j$ be the natural inclusion map from $X$ into the elements of $l_{p}(X)$
which are zero except in the $j$th coordinate and $P_j$ be the projection from
$l_{q}(Y)$ onto $Y$ given by choosing the $j$th coordinate.
For all $j$,  let
$x_j\in B_X$, such that $$\|P_j T E_j x_j\|= \|D_j
x_j\|\ge \|D_j\|/2.$$
Because we began with a compact operator, $\{T E_j x_j\}$
is relatively compact in $l_q(Y).$ If $1\le
q <\infty$, $\|P_j T E_j
x_j\|$ converges to $0$ because $T(B_{l_p(X)})$ is relatively compact and
thus $\sum \|P_jy\|_Y^q$ converges uniformly for
$y \in T(B_{l_p(X)})$. If $q=\infty$, then $p>1$ and $(E_j x_j)$ converges
weakly to $0$. Because $T$ is compact, $(TE_j x_j)$ converges in norm to $0$.
Therefore
the limit of the norms of the operators $D_j$ must be $0$. Conversely,
any sequence of operators $(T_i)$ with $T_i \in {\mathcal K} (X,Y)$ for all $i$
and $\lim \|T_i\|=0$, induces a diagonal operator in ${\mathcal K} (l_{p}(X),
l_{q}(Y))$, by $D (x_j)=(T_j x_j)$. Clearly each truncation $D^{(n)}$ of $D$,
$$D^{(n)} (x_j)=(T_1 x_1,\dots,T_n x_n,0,0,\dots),$$ is compact and $D^{(n)}$
converges to $D$ in norm.

\end{proof}

\begin{remark} The conditions on $p$ and $q$ are not necessary
for the proof that the space of diagonals of the compact operators is
the range of a contractive map.  In fact $\ell_p$ and $\ell_q$ can
be replaced by spaces with unconditional basis.  The computation of
the bound on the norm of the diagonal operator requires that the norm
on the domain dominate  the norm on the range. In addition to prove
compactness of the diagonal we used the fact that the norm of the tail
of an element in $l_{q}(Y) \cap T(B_{l_p(X)})$ goes to zero. If $p=1$
and $q=\infty$, this may fail and the diagonal of a compact operator
may not be compact. An example of this is the one dimensional operator
$T:\ell_1 \rightarrow \ell_\infty$ defined by $T  (a_j)  = (\sum a_j)
1_{\mathbb N}$. The corresponding diagonal operator is the inclusion
map, $J (a_j) = \sum a_j 1_{\{j\}}.$ \end{remark}

\begin{theorem}\label {isok} Let $X$ and $Y$ be Banach spaces and $1\leq
p < \infty$ and $p \leq q\le \infty$ with  $(p, q) \neq (1, \infty)$. Then we have
$${\mathcal K}(l_{p}(X), l_{q}(Y)) \sim C(\omega,{\mathcal K}(l_{p}(X),
l_{q}(Y))).$$
\end{theorem}

\begin{proof}  By Lemma \ref{c0sum}  with $X_1=l_{p}(X)$
and $Y_1=l_{q}(Y)$ in place of $X$ and $Y$,
we see that $$C_{0}(\omega,{\mathcal K}(l_{p}(X), l_{q}(Y))) =
C_0(\omega,{\mathcal K}(X_1,Y_1)) 
\st{c}{\hr}   {\mathcal K}(X_1,Y_1)={\mathcal K}(l_{p}(X), l_{q}(Y)).$$
Therefore by Pe\l czy\'nski
decomposition method \cite[page 54]{LT} and \ref{BPiso} we infer
$${\mathcal K}(l_{p}(X), l_{q}(Y)) \sim C_{0}(\omega,{\mathcal K}(l_{p}(X), l_{q}(Y)))
\sim C(\omega,{\mathcal K}(l_{p}(X), l_{q}(Y)).$$
\end{proof}

\begin{corollary}\label{iso}
$C(\beta \mathbb N,\ell_q)$ is isomorphic to 
$C(\omega \times\beta \mathbb N,\ell_q),$ for $1\le q <\infty.$
\end{corollary}

\begin{proof} Let $p=1,$ $X=\ell_1$ and $Y=\ell_q$ in the previous result.
Recall that $\mathcal K(\ell_1,\ell_q)$ is isomorphic to $C(\beta \mathbb
N,\ell_q)$. Thus $$C(\beta \mathbb N,\ell_q) \sim C(\omega,C(\beta \mathbb
N,\ell_q)) \sim C(\omega \times\beta \mathbb N,\ell_q).$$
\end{proof}

An analysis of the proof for the special case $p=1$ shows that we can prove
a version of the Cembranos-Freniche result for the case $K=\beta
\mathbb N.$

\begin{proposition}\label{BNextCF}
Suppose that $X$ is a Banach space, $K<\infty$, and
$(P_j)$ is a sequence of projections defined on $X$ with range $X_j$ and
$\|P_j\| \le K$ for all $j$ such that $\lim_j \|P_j x\|=0$ for all $x \in
X.$ Then $(\sum_j C(\beta \mathbb N,X_j))_{c_0}$ is complemented in
$C(\beta \mathbb N,X)$.
\end{proposition}

\begin{proof}
Let $\{N_j\}$ be a partition of $\mathbb N$ into countably many disjoint
infinite sets. Define an operator $P$ on $C(\beta \mathbb N, X)$ by
$(Pf)(n)=P_j(f(n))$ for all $f \in C(\beta \mathbb N, X)$,
for all $n \in N_j$, $j =1,2,\dots.$ We will show that $Pf$ has norm relatively
compact range and thus extends to a continuous function on $\beta
\mathbb N.$ The bound $K$ on the norms of the projections shows that the
range of $Pf$ is bounded in $X$. To see that the range is totally bounded,
let $\epsilon > 0$ and $\{x_m:m\in M\}$ be a finite
$(\epsilon/4)K^{-1}$-net in $f(\beta \mathbb N).$ For each $m \in M$,
there is an integer $J_m$  such that for all $j \ge J_m$, $$\|P_j
x_m\|<\epsilon/4.$$ 
Let $J= \max_m J_m.$ If $j \ge J$ and $x \in f(\beta
\mathbb N)$ then 
$$\|P_j x \| \le \min_m (\|P_j x_m \| +\|P_j(x_m-x)\|)<\epsilon/2.$$ Thus the
range of $Pf$ is contained in $$\cup_{j\le J} P_j(f(\beta \mathbb N))
\cup \frac{\epsilon}{2} B_X,$$ and an $\epsilon/2$-net in the compact set
$\cup_{j\le J} P_j(f(\beta \mathbb N))$ will yield an $\epsilon$-net.

Clearly $P$ will be linear, bounded and the identity on $C(\beta N_j,X_j)$
for all $j$. Moreover the argument above shows that $$Pf\in (\sum_j C(\beta
N_j,X_j))_{c_0}.$$
\end{proof}

\begin{remark} The results in these sections allow us to point out some limitations of our
approach if one considers more general classification problems for the
spaces
$C(K,X).$ Notice that $C(\beta \mathbb N ,\ell_2 \oplus c_0)$ is isomorphic to  $$C(\beta
\mathbb N ,\ell_2) \oplus C(\beta \mathbb N ,c_0) \sim C(\omega \times
\beta \mathbb N ,\ell_2) \oplus C(\omega \times \beta \mathbb N) 
\sim C(\beta
\mathbb N ,\ell_2).$$
If we instead use $C(\omega^\omega)$, we get a different outcome.
$$C(\beta \mathbb N ,\ell_2 \oplus C(\omega^\omega)) \sim 
C(\beta
\mathbb N ,\ell_2)\oplus C(\beta \mathbb N ,C(\omega^\omega)).$$
This space is not isomorphic to $C(\beta \mathbb N ,\ell_2 )$. It is
complemented in $C(\beta \mathbb N \times \omega^\omega,\ell_2)$ but it
does not seem likely that it contains $C(\omega^\omega,\ell_2)$ as a
complemented subspace. It also seems doubtful that there is any compact
Hausdorff space $K$ such that $C(\beta \mathbb N ,\ell_2 \oplus
C(\omega^\omega))$ is isomorphic to $C(K,\ell_2).$
\end{remark}

\section{Open questions} \label{prob}
We end this paper by  stating some questions which it raises.  We do not know whether the statement of our main result (Theorem \ref{main}) remains true in the case where   $X=l_{\infty}$, that is,
\begin{problem}\label{p1} Let $K_{1}$ and $K_{2}$  be infinite  compact metric
spaces. Does it follow that
$$C(\beta {\mathbb N} \times  \beta {\mathbb N}   \times K_{1}) \sim C(\beta {\mathbb N} \times \beta {\mathbb N} \times
K_{2}) \Longrightarrow C(K_{1}) \sim C(K_{2})?$$
\end{problem}
Notice  that with $X=\mathbb R$ in Theorem \ref{main} we have (see also \cite[Theorem 5.7]{GO})
$$C(\beta {\mathbb N}   \times K_{1}) \sim C(\beta {\mathbb N} \times
K_{2}) \Longrightarrow C(K_{1}) \sim C(K_{2}),$$
for any infinite compact metric spaces $K_{1}$ and $K_{2}$.

The case $K_1=\omega$ and $K_2=\{1\}$ of Problem \ref{p1}
is the statement of Theorem \ref{iso}
when $q=\infty$, that is,

\begin{problem}\label{p2} Is it true that
$$C(\omega \times \beta {\mathbb N} \times  \beta {\mathbb N}) \sim C(\beta {\mathbb N} \times  \beta {\mathbb N})?$$
\end{problem}

This is a special case of the following question for which Theorem
\ref{iso} gives some answers:

\begin{problem} For which infinite dimensional Banach spaces $X$ is
$$C(\omega \times \beta \mathbb N, X)  \sim C(\beta \mathbb N, X)?$$
\end{problem}

Of course if $X$ is a finite dimensional space, this is false. If
$X$ is isomorphic to $C(\omega,Y)$, then $$C(\omega \times \beta \mathbb N, X) \sim
C(\beta \mathbb N,C(\omega \times \omega ,Y)) \sim C(\beta \mathbb
N,C(\omega,Y)),$$ and thus such $X$ are examples.

One way to approach Problem \ref{p2} is to study the isomorphic
classification
of the complemented subspaces of
$C(\beta {\mathbb N} \times  \beta {\mathbb N})$. Thus, it would
be interesting to solve the following intriguing problem  which is
a particular case of the well known complemented subspace problem for
$C(K)$ spaces, see for instance \cite[section 5]{R}

\begin{problem} Let $X$ be  a complemented subspace of $C(\beta {\mathbb N} \times  \beta {\mathbb N})$. Suppose that $X$ is  infinite dimensional  separable space. Is $X$ isomorphic to $c_{0}$?
\end{problem}

In particular, the other separable $C(K)$ spaces would  be eliminated if
the answer to the following is no.

\begin{problem} Is it true that
 $C(\omega^\omega) \st{c}{\hr} C(\beta {\mathbb N} \times  \beta {\mathbb N})?$
\end{problem}

Finally, observe that  Theorem \ref{com}  leads naturally to the following  problem which is in connection with the Cembranos and Freniche's theorem (Theorem \ref{CF})

\begin{problem} Suppose that $X$ is a Banach space and $K$ is an infinite compact space. Is it true that
\begin{enumerate}
\item [(1)] $C(\omega^\omega) \st{c}{\hr}  C(K, X)     \Longrightarrow     c_{0}    \st{c}{\hr} C(K) \ \hbox{or} \  c_{0} \hr X?$
\end{enumerate}

\begin{enumerate}
\item [(2)] $C(\omega^\omega) \st{c}{\hr}  C(\beta \mathbb N, X)     \Longrightarrow     c_{0}    \st{c}{\hr} X?$
\end{enumerate}
\end{problem}

If $X$ is separable, then $c_0$ is always complemented so the latter is
true by Theorem \ref{com}. As noted after the proof of that theorem
variations of this problem with either $C(\omega^\omega)$ or uniformly
complemented copies of $(\omega^n)$ could also be considered.

It is possible that the proper context for this line of investigation is
actually injective tensor products.

\begin{problem} Suppose that $X$ and $Y$ are Banach spaces, $\alpha>0$
and $\alpha_n \uparrow \omega^\alpha$ are ordinals, and
$C(\omega^{\omega^\alpha})$ isomorphic to a (complemented) subspace of
$X \stackrel{\scriptscriptstyle
\vee}{\otimes} Y$, is
\begin{itemize}
\item{} $C(\omega^{\omega^\alpha})$ isomorphic to a
(complemented) subspace of $X$ or $Y$

or

\item{} $C(\omega^{\alpha_n})$ uniformly isomorphic to (complemented)
subspace of one of $X$ and $Y$ and $c_0$ isomorphic to a subspace of the
other?
\end{itemize}
\end{problem}


\bibliographystyle{amsplain}

\end{document}